\documentclass[a4paper,11pt]{article}

\usepackage[]{graphicx}
\usepackage[]{amsmath,amssymb}
\usepackage[]{amsthm,enumerate}
\usepackage{verbatim,bm,comment}
\usepackage{color}
\usepackage{comment}
\usepackage{a4wide}

\newtheorem{theorem}{Theorem}[section]
\newtheorem{lemma}[theorem]{Lemma}
\newtheorem{proposition}[theorem]{Proposition}
\newtheorem{corollary}[theorem]{Corollary}
\theoremstyle{definition}

\newcommand{\defn}[1]{{\em #1}}
\theoremstyle{remark}
\newtheorem{remark}[theorem]{Remark}

\title{Non-commutative association schemes and their fusion association schemes} 
\date{
\today
}

\author{
 Hadi Kharaghani\thanks{Department of Mathematics and Computer Science, University of Lethbridge,
Lethbridge, Alberta, T1K 3M4, Canada. \texttt{kharaghani@uleth.ca}} 
\and  
 Sho Suda\thanks{Department of Mathematics Education,  Aichi University of Education, Kariya, Aichi 448-8542, Japan. \texttt{suda@auecc.aichi-edu.ac.jp}}
}
\begin{document}

\maketitle

%\begin{center}
%2010 Mathematical Subject Classification: Primary 05E30; Secondary 05B30
%\end{center}
%%%%%%%%%%%%%%%%%%%%%%%%%%%%%%%%%%%%%%%%%%%%%%%%%%%%%%%%%%%%%%%%%%%%%%%%%%%%%%%%%%%
%%%%%%%%%%%%%%%%%%%%%%%%%%%%%%%%%%%%%%%%%%%%%%%%%%%%%%%%%%%%%%%%%%%%%%%%%%%%%%%%%%
\abstract{
We %establish an analog of Bannai-Muzychuk criterion
 give a sufficient condition for a non-commutative association scheme to have a fusion association scheme, and construct non-commutative association schemes from symmetric balanced generalized weighing matrices and generalized Hadamard matrices. 
We then apply the criterion to these non-commutative association schemes to obtain symmetric fusion association schemes.  
%We construct a set of divisible designs with symmetric incidence matrices which share some graphs such as the complete graphs or complete multipartite graphs from generalized Hadamard matrices or symmetric balanced generalized matrices with constant diagonals. 
%We also show that non-commutative association schemes are obtained. 
}

\section{Introduction}
Association schemes are considered as an abstraction of the centralizer of transitive permutation groups, and can be described as the subalgebra of the matrix algebra generated by the disjoint $(0,1)$-matrices which are closed under the transposition and their sum equals to the all-ones matrix \cite{B}, \cite{Z}.   
Much of interest is focused on the case of multiplicity free transitive permutation groups. In such cases, the corresponding association schemes are commutative.  
In the present paper, we consider non-commutative association schemes obtained from some combinatorial objects such as symmetric balanced generalized weighing matrices and generalized Hadamard matrices. 
%It is defined as a set of disjoint $(0,1)$-matrices satisfying the property that 

Kharaghani and Torabi \cite{KT} showed that for any prime power $q$, the edge set of complete graph $K_{q^3+q^2+q+1}$ is decomposed into $q+1$ strongly regular graphs sharing $q^2+1$ disjoint cliques. 
The decomposition is based on symmetric balanced generalized weighing matrices $BGW(q^2+1,q^2,q^2-1)$ with zero diagonal entries  over a cyclic group of order $q+1$, see \cite{GM} for details. Motivated by this decomposition, Klin, Reichard and Woldar \cite{KRW} defined the concept of Siamese objects as a partition of the edge set of the complete graph, and studied it from the view point of graph theory and group theory. 
In particular, it was shown that an action of $PGL(2,q)$ yields a non-commutative association scheme. 

In this paper it is shown that non-commutative association schemes are obtained from the following objects:
\begin{itemize}
\item any symmetric balanced generalized weighing matrix $BGW(n+1,n,n-1)$ with zero diagonal entries over a cyclic group $C_m$ of order $m$,
\item a generalized Hadamard matrix attached to finite fields. 
\end{itemize}
Our first example of non-commutative association scheme is obtained from the $BGW(n+1,n,n-1)$ over the cyclic group $C_m$ with $(n,m)=(q^2,q+1)$,  includes Kharaghani and Torabi's work, and has the  parameters  of the non-commutative association scheme obtained by Klin, Reichard and Woldar. 
We also establish a sufficient condition for a non-commutative association scheme to possess a fusion association scheme, and thus obtain an analog of a part of result by Bannai \cite{B} and Muzychuk \cite{M}. By applying this criterion to our non-commutative association scheme we obtain some symmetric fusion association schemes. 
Finally, the Wedderburn decomposition (or character table) and the eigenmatrices of the association schemes are explicitly determined.
  
\section{Preliminaries}
Throughout this paper, $I_n,J_n$ denote the identity matrix of order $n$, the all-ones matrix of order $n$ respectively.

%\subsection{Association schemes}
Let $d$ be a positive integer. 
Let $X$ be a finite set of size $v$ and $R_i$ ($i\in\{0,1,\ldots,d\}$) be a nonempty subset of $X\times X$. 
The \emph{adjacency matrix} $A_i$ of the graph with vertex set $X$ and edge set $R_i$ is a $v\times v$ $(0,1)$-matrix with rows and columns indexed by $X$ such that $(A_i)_{xy}=1$ if $(x,y)\in R_i$ and $(A_i)_{xy}=0$ otherwise. 
An \emph{association scheme} of $d$-class is a pair $(X,\{R_i\}_{i=0}^d)$ satisfying the following:
\begin{enumerate}[(i)]
\item $A_0=I_{v}$.
\item $\sum_{i=0}^d A_i = J_{v}$.
\item $A_i^\top \in\{A_1,\ldots,A_d\}$ for any $i\in\{1,\ldots,d\}$.
\item For all $i$ and $j$, $A_i A_j$ is a linear combination of $A_0,A_1,\ldots,A_d$.
\end{enumerate}
We also refer to the set of non-zero $(0,1)$-matrices satisfying (i)--(iv) as an association scheme. 
An association scheme is \emph{symmetric} if $A_i^\top=A_i$ holds for any $i$. 
An association scheme is \emph{commutative} if $A_iA_j=A_jA_i$ holds for any $i,j$, \emph{non-commutative} otherwise. 
Non-commutative association schemes are also known as \emph{homogeneous coherent configurations} \cite{H}. 
Note that symmetric association schemes are commutative association schemes by (iv). 
For a symmetric association scheme $(X,\{R_i\}_{i=0}^2)$ of class $2$,  the graph with adjacency matrix $A_i$ for $i\in\{1,2\}$ is said to be a \emph{strongly regular graph}.

Let $\{A_0,A_1,\ldots,A_d\}$ be a non-commutative association scheme.  
The vector space over $\mathbb{C}$ spanned by the $A_i$'s forms a non-commutative algebra, denoted by $\mathcal{A}$ and called the \emph{adjacency algebra}.
Since the algebra $\mathcal{A}$ is semisimple, the adjacency algebra is isomorphic to $\oplus_{k=1}^n\text{Mat}_{d_k}(\mathbb{C})$ for uniquely determined positive integers $n,d_1,\ldots,d_n$. 
We write $\mathcal{I}=\{(i,j,k)\in\mathbb{N}^3 \mid 1\leq i,j\leq d_k, 1\leq k\leq n\}$ where $\mathbb{N}$ denotes the set of positive integers, and call the set $\mathcal{I}$ the \emph{index set of a dual basis}. 

For $k\in\{1,\ldots,n\}$, let $\varphi_k$ be an irreducible representation from $\mathcal{A}$ to $\text{Mat}_{d_k}(\mathbb{C})$. 
Despite the fact that the entries of the image of irreducible representations are not uniquely determined,  the \defn{character table} defined below is uniquely determined. 
The character table $T$ is defined to be an $n\times (d+1)$ matrix with $(k,i)$ entry equal to $\mathrm{tr}(\varphi_k(A_i))$, where $\mathrm{tr}$ denotes the trace.     

%For $k\in\{1,\ldots,m\},i,j\in\{1,\ldots,d_k\}$, 
The following is due to Higman \cite{H75}.  
Let $E_{i,j}^{(k)}$ ($(i,j,k)\in\mathcal{I}$) be a basis of $\mathcal{A}$ such that $\varphi_k(E_{i,j}^{(k)})\in\text{Mat}_{d_k}(\mathbb{C})$, $E_{i,j}^{(k)}E_{i',j'}^{(k')}=\delta_{k,k'}\delta_{j,i'}E_{i,j'}^{(k)}$ and ${E_{i,j}^{(k)}}^*=E_{j,i}^{(k)}$, where $*$ denotes the transpose conjugate.   
Since $A_l$ ($l\in\{0,1,\ldots,d\}$) and $E_{i,j}^{(k)}$ ($(i,j,k)\in\mathcal{I}$) are bases of the adjacency algebra, there exist complex numbers $p_{(i,j),l}^{(k)}$ and $q_{l,(i,j)}^{(k)}$ such that 
\begin{align*}
%A_l=\sum_{k=1}^m\sum_{i,j=1}^{d_k} p_{(i,j),l}^{(k)}E_{i,j}^{(k)},\quad
A_l=\sum_{(i,j,k)\in\mathcal{I}} p_{(i,j),l}^{(k)}E_{i,j}^{(k)},\quad
E_{i,j}^{(k)}=\frac{1}{v}\sum_{l=0}^{d} q_{l,(i,j)}^{(k)}A_l.
\end{align*}
For $k\in\{1,\ldots,n\}$, set $P_k=(p_{(i,j),l}^{(k)})$ and $Q_k=(q_{l,(i,j)}^{(k)})$, where $(i,j)$ runs over $\{(a,b)\mid 1\leq a,b\leq d_k\}$ and $l$ runs over $\{0,1,\ldots,d\}$. 
Note that the ordering of indices of rows of $P_k$ and columns of $Q_k$ is the lexicographical order. 
We then define  $(d+1)\times (d+1)$ matrices $P$ and $Q$ by 
\begin{align*}
P=\begin{pmatrix}
P_1\\
P_2\\
\vdots\\
P_n
\end{pmatrix}, \quad 
Q=\begin{pmatrix}
Q_1& Q_2& \cdots & Q_n
\end{pmatrix}.
\end{align*}
In order to derive the character table from the matrix $Q$, we prepare the following lemma. 
\begin{lemma}\label{lem:m}
Let $k\in\{1,\ldots,n\}$. 
\begin{enumerate}
\item For any $i,j,i',j'\in\{1,\ldots,d_k\}$, $\mathrm{rank}{E_{i,j}^{(k)}}=\mathrm{rank}E_{i',j'}^{(k)}$. 
\item For any $i,j\in\{1,\ldots,d_k\}$, $\mathrm{tr}E_{i,j}^{(k)}=\delta_{i,j}\mathrm{rank}{E_{i,j}^{(k)}}$. 
\end{enumerate}
\end{lemma}
\begin{proof}
Let $i,j\in\{1,\ldots,d_k\}$.

(1): 
Since ${E_{i,j}^{(k)}}^*=E_{j,i}^{(k)}$, $\mathrm{rank}{E_{i,j}^{(k)}}=\mathrm{rank}E_{j,i}^{(k)}$. 
Next we use $E_{i,j}^{(k)}E_{j,i}^{(k)}=E_{i,i}^{(k)}$ to obtain $\mathrm{rank}{E_{i,i}^{(k)}}=\mathrm{rank}E_{i,j}^{(k)}$. 
This proves (1). 

(2): Since $E_{i,i}^{(k)}$ is an idempotent, $\mathrm{tr}E_{i,i}^{(k)}=\mathrm{rank}{E_{i,i}^{(k)}}$. 
Next assume that $i$ is not equal to $j$. Then, taking the trace of $E_{i,j}^{(k)}=E_{i,i}^{(k)}E_{i,j}^{(k)}$, we have 
\begin{align*} 
\mathrm{tr}E_{i,j}^{(k)}=\mathrm{tr}(E_{i,i}^{(k)}E_{i,j}^{(k)})=\mathrm{tr}(E_{i,j}^{(k)}E_{i,i}^{(k)})=0. 
\end{align*} 
Therefore we have $\mathrm{tr}E_{i,j}^{(k)}=\delta_{i,j}\mathrm{rank}{E_{i,j}^{(k)}}$.
\end{proof}
By Lemma~\ref{lem:m}, we let $m_k=\mathrm{rank}{E_{i,j}^{(k)}}$ for $k\in\{1,\ldots,n\}$ and any $i,j\in\{1,\ldots,d_k\}$. 
We also let $v_l=\frac{1}{v}\tau(A_l)$, where $\tau$ denotes the function assigning a matrix to the sum of the entries. 

The following proposition shows how to derive the character table from the matrix $Q$. 
The proof is based on the same idea as in \cite[Theorem 3.5 (i)]{BI}.
\begin{proposition}\label{prop:ct}\rm 
\begin{enumerate}
\item  For each $k$, the $k$-th row of $T$ is the scalar multiple by $m_k$ of the sum of the rows of $P_k$ corresponding to all indices $(i,i)$. 
\item $\Delta_m P= \bar{Q} \Delta_v$ holds, where $\bar{Q}$ is the complex conjugate of $Q$,  $\Delta_m$ is the diagonal matrix indexed by $\mathcal{I}$ with $(i,j,k)$-th entry equal to $m_k$ and $\Delta_v$ is the diagonal matrix indexed by $\{0,1,\ldots,d\}$ with $i$-th entry equal to $v_i$. 
\end{enumerate}
\end{proposition}
\begin{proof}
(1): Since $\varphi_k(A_l)=\sum_{i,j=1}^{d_k}p_{(i,j),l}^{(k)}E_{i,j}^{(k)}$, we use  Lemma~\ref{lem:m} to obtain the following; 
\begin{align*}
\mathrm{tr}(\varphi_k(A_l))=\sum_{i,j=1}^{d_k}p_{(i,j),l}^{(k)}\mathrm{tr}(E_{i,j}^{(k)})=\sum_{i=1}^{d_k}p_{(i,i),l}^{(k)}m_k,\end{align*} 
from which we obtain the desired result. 

(2): 
Calculating $\textrm{tr}(A_lE_{i,j}^{(k)})$ in two ways as follows. On the one hand, by Lemma~\ref{lem:m} 
\begin{align*}
\textrm{tr}(A_lE_{i,j}^{(k)})&=\textrm{tr}(\sum_{i'=1}^{d_k} p_{(i',i),l}^{(k)}E_{i',j}^{(k)})\\
&=p_{(j,i),l}^{(k)}m_k.
\end{align*}
On the other hand, 
let $l'$ be such that $A_l^\top=A_{l'}$ and  $\circ$ denote the entry-wise product for matrices. 
Since $\textrm{tr}(AB)=\tau(A^\top\circ B)$ and $v_{l'}=v_l$, we have 
\begin{align*}
\textrm{tr}(A_lE_{i,j}^{(k)})&=\tau(A_{l'}\circ E_{i,j}^{(k)})\\
&=\tau(\frac{1}{v}q_{l',(i,j)}^{(k)}A_{l'})\\
&=v_{l'}  q_{l',(i,j)}^{(k)}\\
&=v_l  \overline{q_{l,(j,i)}^{(k)}},  
\end{align*}
where the last equality follows from ${E_{i,j}^{(k)}}^*=E_{j,i}^{(k)}$. 
Thus we obtain the desired result. 
\end{proof}

%%%%%%%%%%%%%%%%%%%%%%%%%%%%%%%%%%%%%%%%%%%%%%%%%%%%%%%%%%%%%%%%%%%%%%%%%%%%%%%%%%%%%%%%%%%%%%
\section{Bannai-Muzychuk criterion for non-commutative association schemes}
Let $(X,\{R_i\}_{i=0}^d)$ be an association scheme. 
Let $\{\Lambda_0,\Lambda_1,\ldots,\Lambda_e\}$ be a partition of $\{0,1,\ldots,d\}$ such that $\Lambda_0=\{0\}$. 
If $(X,\{\cup_{l'\in\Lambda_l}R_{l'}\}_{l=0}^e)$ is an association scheme, then it is called a \emph{fusion scheme} of the association scheme $(X,\{R_i\}_{i=0}^d)$. 
The following is a result obtained by Bannai \cite{B} and by Muzychuk \cite{M} independently. It is a criterion characterizing the fusion scheme in terms of eigenmatrices for commutative schemes. 
For commutative association schemes, the index set of a dual basis $\mathcal{I}$ is of the form $\mathcal{I}=\{(1,1,k)\mid k=1,\ldots,d+1\}$. 
Thus we use the standard notion of the entries of the first eigenmatrix $P=(p_{i,j})_{i,j=0}^{d+1}$: $p_{i,j}=p_{(1,1),j}^{(i)}$ for $i,j\in\{0,1,\ldots,d\}$. 
\begin{theorem}\label{thm:BMc}\rm 
Let $(X,\{R_i\}_{i=0}^d)$ be a commutative association scheme. Let $\{\Lambda_0,\Lambda_1,\ldots,\Lambda_e\}$ be a partition of $\{0,1,\ldots,d\}$ such that $\Lambda_0=\{0\}$. 
Then the following are equivalent. 
\begin{enumerate}
\item $(X,\{\bigcup_{l'\in\Lambda_l}R_{l'}\}_{l=0}^e)$ is a fusion scheme of $(X,\{R_i\}_{i=0}^d)$.
\item  (a) For any $l\in\{0,1,\ldots,e\}$ there exists $k\in\{0,1,\ldots,e\}$ such that $\sum_{l'\in\Lambda_l}A_{l'}^\top=\sum_{k'\in\Lambda_k}A_{k'}$, and (b) there exists a partition $\{\Delta_k\mid k=0,1,\ldots,e\}$ of $\{0,1,\ldots,n\}$ such that $\Delta_0=\{0\}$ and  $\sum_{l'\in\Lambda_{l}}p_{k',l'}$ does not depend on the choice of $k'\in \Delta_k$.
\end{enumerate}
\end{theorem}

In the following, we give a sufficient condition for a non-commutative association scheme to have a fusion scheme, an analog of one way of Theorem~\ref{thm:BMc} in non-commutative association schemes.% obtained by Bannai \cite{B} and independently by Muzychuk \cite{M}. It is a criterion characterizing the fusion scheme in terms of eigenmatrices for commutative schemes. 

Note that in the following lemma we use a specific partition of $\mathcal{I}=\{(i,j,k)\in\mathbb{N}^3 \mid 1\leq i,j\leq d_k, 1\leq k\leq n\}$, for the index set of a dual basis. 
For each $k\in\{1,\ldots,n\}$, let $\{I_1^{(k)},\ldots,I_{f_k}^{(k)}\}$ be a partition of $\{1,\ldots,d_k\}$.   
For $a,b\in\{1,\ldots,f_k\}$, define $\mathcal{I}_{a,b}^{(k)}=I_a^{(k)}\times I_b^{(k)}$.  
Then we call the partitions $\mathcal{I}_{a,b}^{(k)}$ ($k\in\{1,\ldots,n\}, a,b\in \{1,\ldots,f_k\}$) of $\{1,\ldots,d_k\}^2$ \emph{canonical}. 
\begin{theorem}\label{thm:BM}\rm 
Let $(X,\{R_i\}_{i=0}^d)$ be an association scheme. 
Let $\{\Lambda_0,\Lambda_1,\ldots,\Lambda_e\}$ be a partition of $\{0,1,\ldots,d\}$ such that $\Lambda_0=\{0\}$. 
Assume 
\begin{enumerate}[(a)]
\item For any $l\in\{0,1,\ldots,e\}$ there exists $k\in\{0,1,\ldots,e\}$ such that $\sum_{l'\in\Lambda_l}A_{l'}^\top=\sum_{k'\in\Lambda_k}A_{k'}$, and 
\item there exists a canonical partition $\mathcal{I}_{a,b}^{(k)}$ ($a,b\in \{1,\ldots,f_k\}$) of $\{1,\ldots,d_k\}^2$ for any $k\in\{1,\ldots,n\}$ such that $\sum_{k=1}^{n}f_k^2=e+1$ and  
$\sum_{l'\in\Lambda_{l}}p_{(i,j),l'}^{(k)}$ does not depend on the choice of $(i,j)\in \mathcal{I}_{a,b}^{(k)}$.  
\end{enumerate}
Then $(X,\{\bigcup_{l'\in\Lambda_l}R_{l'}\}_{l=0}^e)$ is a fusion scheme of $(X,\{R_i\}_{i=0}^d)$.
\end{theorem}
\begin{proof}
%Let $\mathcal{A}$ be the adjacency algebra of the association scheme of $(X,\{R_i\}_{i=0}^d)$. 

%(1)$\Rightarrow$(2): Clearly (a) holds. 
%Since $\sum_{l'\in\Lambda_{l}}A_{l'}$ ($l\in\{0,1,\ldots,e\}$) form a subalgebra $\mathcal{A}'$ of the adjacency algebra $\mathcal{A}$,  Lemma~\ref{lem:subalg} implies that there exists a canonical partition $\mathcal{I}_{a,b}^{(k)}$ ($k\in\{1,\ldots,n\}, a,b\in \{1,\ldots,f_k\}$) of $\{1,\ldots,d_k\}^2$ such that  $\sum_{(i,j)\in \mathcal{I}_{a,b}^{(k)}}E_{i,j}^{(k)}$  ($k\in\{1,\ldots,n\}, a,b\in \{1,\ldots,f_k\}$) form a dual basis of $\mathcal{A}'$. 
%Considering the dimension of $\mathcal{A}'$, $\sum_{k=0}^e f_k^2=e+1$ holds. 
%Furthermore, for $l\in\{0,1,\ldots,e\}$ 
%\begin{align*}
%\sum_{l'\in\Lambda_{l}}A_{l'}&=\sum_{l'\in\Lambda_{l}}\sum_{(i,j,k)\in\mathcal{I}}p_{(i,j),l'}^{(k)}E_{i,j}^{(k)}
%=\sum_{k=1}^n\sum_{a,b=1}^{f_k} \sum_{(i,j)\in \mathcal{I}_{a,b}^{(k)}}  \left(\sum_{l'\in\Lambda_{l}}p_{(i,j),l'}^{(k)}\right)E_{i,j}^{(k)},  
%\end{align*}
%which must be in $\mathcal{A}'$. 
%Thus the coefficient of $E_{i,j}^{(k)}$ does not depend on the choice of $(i,j)\in \mathcal{I}_{a,b}^{(k)}$. 
%Therefore (b) holds. 
%(2)$\Rightarrow$(1): 
Let $\mathcal{A}'$ be the vector space spanned by $\sum_{l'\in \Lambda_l}A_{l'}$ for $l\in\{0,1,\ldots,e\}$. Then $\mathcal{A}'$ is closed under the transposition by (a). 
Consider a subalgebra $\mathcal{A}''$ generated by the matrices $\sum_{(i,j)\in \mathcal{I}_{a,b}^{(k)}}E_{i,j}^{(k)}$ ($k\in\{1,\ldots,n\},a,b\in \{1,\ldots,f_k\}$). 
%Observe that $\mathcal{A}''$ is closed under  the matrix multiplication. 
Letting $c_{(i,j),l}^{(k)}=\sum_{l'\in\Lambda_{l}}p_{(i,j),l'}^{(k)}$, we have 
\begin{align*}
\sum_{k=1}^n\sum_{a,b=1}^{f_k} \sum_{(i,j)\in \mathcal{I}_{a,b}^{(k)}}  c_{(i,j),l}^{(k)}E_{i,j}^{(k)}
&=\sum_{k=1}^n\sum_{a,b=1}^{f_k} \sum_{(i,j)\in \mathcal{I}_{a,b}^{(k)}}  \left(\sum_{l'\in\Lambda_{l}}p_{(i,j),l'}^{(k)}\right)E_{i,j}^{(k)}\\
&=\sum_{l'\in\Lambda_{l}}\sum_{(i,j,k)\in\mathcal{I}}p_{(i,j),l'}^{(k)}E_{i,j}^{(k)}\\
&=\sum_{l'\in\Lambda_{l}}A_{l'}.   
\end{align*}
Thus the subalgebra $\mathcal{A}''$ includes the vector space $\mathcal{A}'$. 
Since the dimension of $\mathcal{A}'$ and $\mathcal{A}''$ as the vector space over $\mathbb{C}$ coincide by the assumption $\sum_{k=1}^{n}f_k^2=e+1$, we have $\mathcal{A}'=\mathcal{A}''$. 
Now since $\mathcal{A}''$ is closed under the matrix multiplication, so is $\mathcal{A}'$.  
Thus $(X,\{\bigcup_{l'\in\Lambda_l}R_{l'}\}_{l=0}^e)$ is an association scheme. 
\end{proof}

%%%%%%%%%%%%%%%%%%%%%%%%%%%%%%%%%%%%%%%%%%%%%%%%%%%%%%%%%%%%%%%%%%%%%%%%%%%%%%%%%%%%%%%%%%%%%%
\section{Symmetric $BGW(n+1,n,n-1)$ with zero diagonal entries over the cyclic group}
In this section, we construct a non-commutative association scheme from any given symmetric $BGW(n+1,n,n-1)$ with zero diagonal entries over a cyclic group. 
For completion we first recall the definition and the existence of symmetric $BGW(q+1,q,q-1)$ with zero diagonal entries over a cyclic group.

Let $G$ be a multiplicatively written finite group. 
A \emph{balanced generalized weighing matrix with parameters $(v, k, \lambda)$ over $G$}, or a \emph{$BGW(v, k, \lambda)$ over $G$}, is a matrix $W = (w_{ij})$
of order $v$ with entries from $G \cup \{0\}$ such that (i) every row of $W$ contains exactly
$k$ nonzero entries and (ii) for any distinct $i, h \in \{1, 2,\ldots, v\}$, every element of $G$ is
contained exactly $\lambda/|G|$ times in the multiset $\{w_{ij} w^{-1}_{hj} \mid 1 \leq j \leq v, w_{ij} \neq 0, w_{hj}\neq 0\}$.
A BGW with $v=k$ is said to be a \emph{generalized Hadamard matrix}, which will be dealt in Section~\ref{sec:GH}. 
The following is a basic result of symmetric balanced generalized weighing matrices.
\begin{lemma}{\rm \cite[Lemma~3]{KT}}
Let $q,m,t$ be positive integers such that $q$ is a prime power, $q=mt+1$, $t$ is even. 
Then there is a symmetric $BGW(q+1,q,q-1)$ with zero diagonal entries over the cyclic group of order $m$.  
\end{lemma}

Let $n,m$ be positive integers. 
Let $W=(w_{i,j})_{i,j=1}^{n+1}$ be any symmetric $BGW(n+1,n,n-1)$ with zero diagonal entries over the cyclic group $C_m=\langle U \rangle$ generated by the circulant matrix $U$ of order $m$ with the first row $(0,1,0,\ldots,0)$. 
We construct a non-commutative association scheme from $W$ as follows.  %any symmetric $BGW(n+1,n,n-1)$ with zero diagonal over the cyclic group of order $m$.   

%For $l\in\{0,1,\ldots,m-1\}$, define $N_l$ to be an $(n+1)m\times (n+1)m$ $(0,1)$-matrix as an $(n+1)\times(n+1)$ block matrix with $(i,j)$-block equal to 

For $l\in\{0,1,\ldots,m-1\}$, define $N_l$ to be an $(n+1)m\times (n+1)m$ $(0,1)$-matrix with $m\times m$ block submatrices such that its $(i,j)$-block equals to 
\begin{align*} 
[N_l]_{ij}&=
\begin{cases}
J_{m} & \text{ if } i=j,\\
w_{ij}U^l R & \text{ if }i\neq j,
\end{cases}
\end{align*}
where $[N_l]_{ij}$ denotes the $(i,j)$-block of $N_l$ and %$i,j\in\{1,\ldots,n+1\}$, 
$R$ is the back diagonal matrix of order $m$. 
Note that $UR=RU^{-1}$, and $w_{ij}U=Uw_{ij}$ since $w_{ij}$ is a power of $U$.  
Then the following holds. 
\begin{theorem}\label{thm:GDD}
\begin{enumerate}
\item The matrices $N_0,\ldots,N_{m-1}$ are symmetric matrices and  share the diagonal blocks $I_{n+1}\otimes J_m$.  
\item For $l\in\{0,1,\ldots,m-1\}$, $N_l(I_{n+1}\otimes J_m)=(I_{n+1}\otimes J_m)N_l=mI_{n+1}\otimes J_m+(J_{n+1}-I_{n+1})\otimes J_m$. 
\item For $l,l'\in\{0,1,\ldots,m-1\}$, 
\begin{align*}
N_l N_{l'}=I_{n+1}\otimes (mJ_m+nU^{l-l'})+(2+\frac{n-1}{m})(J_{n+1}-I_{n+1})\otimes J_m. 
\end{align*}
\end{enumerate}
%In particular, $N_0,\ldots,N_{m-1}$ are divisible designs with parameters $((n+1)m,n+m,m,2+\frac{n-1}{m},n+1,m)$ sharing the diagonal blocks $I_{n+1}\otimes J_m$. 
\end{theorem}
\begin{proof}
(1): 
Letting $i,j$ be distinct, the transpose of $(i,j)$-block of $N_l$ is 
\begin{align*}
(w_{ij}U^l R)^\top=RU^{-l}w_{ij}^{-1}=U^lRw_{ij}^{-1}=U^l w_{ij}R= w_{ij}U^l R, 
\end{align*} 
which is equal to the $(j,i)$-block of $N_l$. 
Thus each $N_l$ is symmetric. 
Since the off-diagonal blocks of $N_l$ and $N_{l'}$ for distinct $l,l'$ are disjoint, $N_l$ and $N_{l'}$ share the diagonal blocks $I_{n+1}\otimes J_m$. 

(2): It follows from the fact that $(w_{ij}U^l R)J_m=J_m(w_{ij}U^l R)=J_m$. 

(3): 
Let $l,l'\in\{0,1,\dots,m-1\}$. 
For $i,j\in\{1,\ldots,n+1\}$ we calculate the $(i,j)$-block of $N_l N_{l'}$ as follows:
\begin{align*}
[N_l N_{l'}]_{ij}&=\sum_{k=1}^{n+1}(\delta_{i,k}J_m+w_{ik}U^l R)(\delta_{j,k}J_m+w_{jk}U^{l'} R)^\top\\
&=\sum_{k=1}^{n+1}\delta_{i,k}\delta_{j,k}mJ_m+\sum_{k=1}^{n+1}\delta_{i,k}w_{jk}^\top J_m+\sum_{k=1}^{n+1}\delta_{j,k}w_{ik} J_m+\sum_{k=1}^{n+1}w_{ik}w_{jk}^\top U^{l-l'} \\
&=\delta_{i,j}mJ_m+w_{ji}^\top J_m+w_{ij} J_m+\left(\delta_{i,j}n I_m+(1-\delta_{i,j})\frac{n-1}{m}J_m\right)U^{l-l'} \\
&=\delta_{i,j}mJ_m+2(1-\delta_{i,j})J_m+\delta_{i,j}n U^{l-l'}+(1-\delta_{i,j})\frac{n-1}{m}J_m \\
&=\begin{cases}
mJ_m+n U^{l-l'} &\text{ if } i=j,\\
(2+\frac{n-1}{m})J_m & \text{ if } i\neq j,
\end{cases}
\end{align*}
where we used  the fact that $w_{ij}$ and $U$ commute in second equality.
\end{proof}

\begin{remark}
\begin{enumerate}
\item A \emph{group divisible design with parameters $(v,k,m,n,\lambda_1,\lambda_2)$} is a pair $(V,\mathcal{B})$ where $V$ is a \emph{point set} of $v$ elements and $\mathcal{B}$ is a \emph{block set} of $k$-element
 subsets of $V$ such that the point set $V$ being decomposed into $m$ classes of size $n$ such that two distinct points from the same class are contained in exactly 
$\lambda_1$ blocks, and two points from different classes are contained in exactly $\lambda_2$ blocks. 
A group divisible design is \emph{symmetric} if its dual is also a group divisible design. 
A $(0,1)$-matrix $A$ is the incidence matrix of a symmetric group divisible design if and only if the $(0,1)$-matrix $A$ satisfies 
\begin{align*}
A A^\top=A^\top A=k I_v+\lambda_1(I_m\otimes J_n-I_v)+\lambda_2(J_v-I_m\otimes J_n).
\end{align*}
%A \emph{divisible design graph} is a graph whose adjacency matrix is the incidence matrix of a group divisible design.  
Theorem~\ref{thm:GDD} shows that each $(0,1)$-matrix $N_l$ is a symmetric group divisible design with parameters $((n+1)m,n+m,n+1,m,m,2+\frac{n-1}{m})$. 
\item If $\lambda:=\lambda_1=\lambda_2$, then the group divisible design is a symmetric  $2$-$(v,k,\lambda)$ design. 
When $m=2+\frac{n-1}{m}$, that is $n=(m-1)^2$, the symmetric group divisible designs $N_0,\ldots,N_{m-1}$ are symmetric $2$-$((n+1)m,n+m,m)$ designs sharing the diagonal blocks $I_{n+1}\otimes J_m$. 
This result is a generalization of \cite[Theorem~5]{KT}. 
\end{enumerate}
\end{remark}

We now construct an association scheme from a symmetric $BGW(n+1,n,n-1)$ with zero diagonal entries over a cyclic group of order $m$. 
Define
\begin{align*}
A_{l,0}=I_{n+1}\otimes U^l,\quad A_{l,1}=N_l-I_{n+1}\otimes J_m
\end{align*}
for $l=0,1,\ldots,m-1$. 
% $A_l=I_{n+1}\otimes U^l$ and $B_l=N_l-I_{n+1}\otimes J_m$ for $l=0,1,\ldots,m-1$. 
\begin{theorem}\label{thm:as11}
The set of matrices $\{A_{l,0},A_{l,1}\mid l=0,1,\ldots,m-1\}$ forms a non-commutative association scheme of class $2m-1$.
\end{theorem}
\begin{proof}
The conditions (i)--(iii) in the definition of association schemes are clearly satisfied. 
We need to show that the condition (iv) in the definition of association schemes is satisfied. 
It is easy to see that  
\begin{align*}
A_{l,0} A_{l',0}=A_{l+l',0},\quad A_{l,0} A_{l',1}&=A_{l+l',1},\quad A_{l,1} A_{l',0}=A_{l-l',1}, 
\end{align*}
where the addition and subtraction of indices are taken in modulo $m$. 
Finally we calculate $A_{l,1} A_{l',1}$ for $l,l'\in\{0,1,\ldots,m-1\}$. 
By Theorem~\ref{thm:GDD}, 
\begin{align*}
A_{l,1}A_{l',1}&=(N_l-I_{n+1}\otimes J_m)(N_{l'}-I_{n+1}\otimes J_m)\\
&=N_l N_{l'}-N_l (I_{n+1}\otimes J_m)-(I_{n+1}\otimes J_m)N_{l'}+mI_{n+1}\otimes J_m\\
&=(I_{n+1}\otimes (mJ_m+n U^{l-l'})+(2+\frac{n-1}{m})(J_{n+1}-I_{n+1})\otimes J_m)\\
&\quad \quad -2(mI_{n+1}\otimes J_m+(J_{n+1}-I_{n+1})\otimes J_m)+mI_{n+1}\otimes J_m\\
&=n I_{n+1}\otimes U^{l-l'}+\frac{n-1}{m}(J_{n+1}-I_{n+1})\otimes J_m\\
&=n A_{l-l',0}+\frac{n-1}{m}(A_{0,1}+\cdots+A_{m-1,1}). 
\end{align*}
Thus the condition (iv) is satisfied. 
\end{proof}

%\subsection{Wedderburn decomposition of the non-commutative association scheme}

We view the cyclic group $C_m$ of order $m$ as the additive group $\mathbb{Z}_m$. 
Let $w=e^{2\pi\sqrt{-1}/m}$. 
For $\alpha,\beta\in\mathbb{Z}_m$, 
the irreducible character denoted $\chi_{\beta}$ is $\chi_{\beta}(\alpha)=w^{\alpha \beta}$. 
The \defn{character table} $K$ of the abelian group $\mathbb{Z}_m$ is an $m\times m$ matrix with rows and columns indexed by the elements of $\mathbb{Z}_m$ with $(\alpha,\beta)$-entry equal to $\chi_{\beta}(\alpha)$. 
Note that $\chi_{\beta}(\alpha)=\chi_{\alpha}(\beta)$. 
Then the Schur orthogonality relation shows that $K K^\top=mI_{m}$.%, namely $K$ is a Hadamard matrix of order $2^m$. 

%Let $C=(w^{ij})_{i,j=0}^{m-1}$ be the Fourier matrix and denote by $\chi_l$ the $(l+1)$-st column of $C$ for $l=0,1,\ldots,m-1$. 

Define $F_{\alpha,0},F_{\alpha,1}$ as 
\begin{align*}
F_{\alpha,0}=\sum_{\gamma\in\mathbb{Z}_m}\chi_\alpha(\gamma)A_{\gamma,0},\quad F_{\alpha,1}=\sum_{\gamma\in\mathbb{Z}_m}\chi_{\alpha}(\gamma)A_{\gamma,1}.
\end{align*}
Using the intersection numbers described in Theorem~\ref{thm:as11}, the following lemma is easy to see. 
\begin{lemma}\label{lem:F}
The matrices $F_{\alpha,0},F_{\alpha,1}$ ($\alpha\in\mathbb{Z}_m$) satisfy the following equations; for $\alpha,\beta\in\mathbb{Z}_m$, 
\begin{align*}
F_{\alpha,0}F_{\beta,0}&=\delta_{\alpha,\beta}mF_{\alpha,0}, \\
F_{\alpha,1}F_{\beta.1}&=\delta_{\alpha,-\beta}nmF_{\alpha,0}+\delta_{\alpha,0}\delta_{\beta,0}m(n-1)F_{0,1},\\
F_{\alpha,0}F_{\beta,1}&=\delta_{\alpha,\beta}mF_{\alpha,1}, \\
F_{\alpha,1}F_{\beta,0}&=\delta_{\alpha,-\beta}mF_{\alpha,1}. 
\end{align*}
\end{lemma}

For $i\in\{0,1,2,3\},j,k\in\{1,2\},\alpha\in \mathbb{Z}_m$, let $E_i,E_{j,k}^{(\alpha)}$ be  
\begin{align*}
E_0&=\frac{1}{(n+1)m}J_{(n+1)m},\\
E_1&=\frac{1}{(n+1)m}(n F_{0,0}-F_{0,1}),\\
E_2&=\frac{1}{2m}(F_{m/2,0}+\frac{1}{\sqrt{n}}F_{m/2,1}),\\
E_3&=\frac{1}{2m}(F_{m/2,0}-\frac{1}{\sqrt{n}}F_{m/2,1}),\\
E_{1,1}^{(\alpha)}&=\frac{1}{m}F_{\alpha,0}, \quad E_{2,2}^{(\alpha)}=\frac{1}{m}F_{-\alpha,0}, \quad E_{1,2}^{(\alpha)}=\frac{1}{m}F_{\alpha,1}, \quad E_{2,1}^{(\alpha)}=\frac{1}{m}F_{-\alpha,1},  
\end{align*}
where $E_2,E_3$ are defined only for the case $m$ even. 
\begin{theorem}\label{thm:bgwas}
%Let $S$ be any subset of $\mathbb{F}_q^*$ such that $S\cup (-S)=\mathbb{F}_q^*$ and $S\cap(-S)=\emptyset$. 
\begin{enumerate}
\item 
If $m$ is even, then the matrices $E_0,E_1,E_2,E_3,E_{1,1}^{(\alpha)},E_{1,2}^{(\alpha)},E_{2,1}^{(\alpha)},E_{2,2}^{(\alpha)}$, $\alpha\in \{1,\ldots,m/2-1\}$, provide the Wedderburn decomposition of the adjacency algebra of the association scheme. 
\item If $m$ is odd, then the matrices $E_0,E_1,E_{1,1}^{(\alpha)},E_{1,2}^{(\alpha)},E_{2,1}^{(\alpha)},E_{2,2}^{(\alpha)}$, $\alpha\in \{1,\ldots,(m-1)/2\}$, provide the Wedderburn decomposition of the adjacency algebra of the association scheme. 
\end{enumerate}
\end{theorem}
\begin{proof}
Both cases readily follow from Lemma~\ref{lem:F}.  
%By Lemma~\ref{lem:f}(i)--(iii), the matrices $E_{i,j}^{(\alpha)}$ satisfy $E_{i,j}^{(\alpha)}E_{j,k}^{(\beta)}=\delta_{\alpha,\beta}\delta_{j,k}E_{i,k}^{(\alpha)}$. 
%Also from Lemma~\ref{lem:f}(iv) and (v), it follows that $E_0,E_1,E_2$ are mutually orthogonal idempotents. 
%Finally the orthogonality between $E_i$ and $E_{j,k}^{(\alpha)}$ follows from Lemma~\ref{lem:f}(vi).
\end{proof}

\begin{remark}
%In Theorem~\ref{thm:bgwas}, the matrices $Q$ are determined by the matrices $Q_k$ below.  
\begin{enumerate}
\item If  $m$ is even, then 
the adjacency algebra is isomorphic to $\oplus_{k=1}^{4+(m-2)/2}\text{Mat}_{d_k}(\mathbb{C})$ where $(d_k)_{k=1}^{4+(m-2)/2}=(1,1,1,1,2,\ldots,2)$
 with  
\begin{align*}
Q_1&=\begin{pmatrix}\chi_0\\ \chi_{0}\end{pmatrix}, Q_2=\begin{pmatrix}n \chi_{0}\\ -\chi_{0}\end{pmatrix},\\
Q_3&=\frac{n+1}{2}\begin{pmatrix}\chi_{m/2}\\ \frac{1}{\sqrt{n}}\chi_{m/2}\end{pmatrix}, Q_4=\frac{n+1}{2}\begin{pmatrix}\chi_{m/2}\\-\frac{1}{\sqrt{n}}\chi_{m/2}\end{pmatrix}, \\
Q_{k+4}&=(n+1)\begin{pmatrix}\chi_k & \bm{0} & \chi_{m-k} & \bm{0} \\ \bm{0} & \chi_k& \bm{0} & \chi_{m-k}\end{pmatrix},
\end{align*} 
for $k=1,\ldots,m/2-1$, where $\bm{0}$ denotes the column zero vector. 
\item If  $m$ is odd, then 
the adjacency algebra is isomorphic to $\oplus_{k=1}^{2+(m-1)/2}\text{Mat}_{d_k}(\mathbb{C})$ where $(d_k)_{k=1}^{2+(m-1)/2}=(1,1,2,\ldots,2)$
 with  
\begin{align*}
Q_1&=\begin{pmatrix}\chi_0\\ \chi_{0}\end{pmatrix}, Q_2=\begin{pmatrix}n \chi_{0}\\ -\chi_{0}\end{pmatrix},\\
Q_{k+2}&=(n+1)\begin{pmatrix}\chi_k & \bm{0} & \chi_{m-k} & \bm{0} \\ \bm{0} & \chi_k& \bm{0} & \chi_{m-k}\end{pmatrix},
\end{align*} 
for $k=1,\ldots,(m-1)/2$.
\end{enumerate}
\end{remark}

As corollaries of Theorem~\ref{thm:as11}, we have the following. 
\begin{corollary}
Let $T$ be the character table of the non-commutative association scheme in Theorem~\ref{thm:as11}.
\begin{enumerate}
\item If $m$ is even, then 
\begin{align*}
T=\begin{pmatrix} 
\chi_0^\top & n\chi_0^\top \\
\chi_0^\top & -\chi_0^\top \\
\chi_{m/2}^\top & \sqrt{n}\chi_{m/2}^\top \\
\chi_{m/2}^\top & -\sqrt{n}\chi_{m/2}^\top \\
\chi_{1}^\top+\chi_{m-1}^\top & \bm{0}^\top \\
\chi_{2}^\top+\chi_{m-2}^\top & \bm{0}^\top \\
\vdots & \vdots \\
\chi_{m/2-1}^\top+\chi_{m/2+1}^\top & \bm{0}^\top \\
\end{pmatrix}.
\end{align*}
\item If $m$ is odd, then 
\begin{align*}
T=\begin{pmatrix} 
\chi_0^\top & n\chi_0^\top \\
\chi_0^\top & -\chi_0^\top \\
\chi_{1}^\top+\chi_{m-1}^\top & \bm{0}^\top \\
\chi_{2}^\top+\chi_{m-2}^\top & \bm{0}^\top \\
\vdots & \vdots \\
\chi_{(m-1)/2}^\top+\chi_{(m+1)/2}^\top & \bm{0}^\top \\
\end{pmatrix}.
\end{align*}

\end{enumerate}
\end{corollary}
\begin{proof}
This follows from Proposition~\ref{prop:ct} and Theorem~\ref{thm:as11}. 
\end{proof}
\begin{corollary}
\begin{enumerate}
\item If $m$ is even, then the set of matrices $\{A_{0,0},A_{i,0}+A_{m-i,0},A_{m/2,0},A_{0,1},A_{i,1}+A_{m-i,1},A_{m/2,1}\mid i=1,\ldots,m/2-1\}$ is a symmetric association scheme with the second eigenmatrix $Q$ 
\begin{align*}
Q=\begin{pmatrix}\chi_0 & n\chi_0 & \frac{n+1}{2}\chi_{m/2} & \frac{n+1}{2}\chi_{m/2}   &\cdots & \frac{n+1}{2}(\chi_{k}+\chi_{m-k}) &\cdots \\ 
\chi_0 & -\chi_0 & \frac{n+1}{2\sqrt{n}}\chi_{m/2} & -\frac{n+1}{2\sqrt{n}}\chi_{m/2}  &\cdots & \frac{n+1}{2}(\chi_{k}+\chi_{m-k}) & \cdots \end{pmatrix},
\end{align*} 
where $k$ runs over $\{1,\ldots,m/2-1\}$. 
\item If $m$ is odd, then the set of matrices $\{A_{0,0},A_{i,0}+A_{m-i,0},A_{0,1},A_{i,1}+A_{m-i,1}\mid i=1,\ldots,(m-1)/2\}$ is a symmetric association scheme with the second eigenmatrix $Q$ 
\begin{align*}
Q=\begin{pmatrix}\chi_0 & n\chi_0 &\cdots & \frac{n+1}{2}(\chi_{k}+\chi_{m-k}) &\cdots \\ 
\chi_0 & -\chi_0  &\cdots & \frac{n+1}{2}(\chi_{k}+\chi_{m-k}) & \cdots \end{pmatrix},
\end{align*} 
where $k$ runs over $\{1,\ldots,(m-1)/2\}$. 
\end{enumerate}
\end{corollary}
\begin{proof}
The results follow from Theorem~\ref{thm:BM}.
%the fact that $\sum_{i,j=1}^{d_k}E_{i,j}^{(k)}$'s form a basis of the algebra generated by the symmetric $(0,1)$-matrices above, and consist of primitive idempotents.  
\end{proof}

%%%%%%%%%%%%%%%%%%%%%%%%%%%%%%%%%%%%%%%%%%%%%%%%%%%%%%%%%%%%%%%%%%%%%%%%%%%%%%%%%%%%%%%%%%%%%%%%%%%%%%%%%%%%%%%%%
\section{Generalized Hadamard matrices}\label{sec:GH}
In \cite{KSS}, Kharaghani, Sasani and Suda considered symmetric association schemes attached to the finite fields of characteristic two. 
In this section, we consider non-commutative association schemes attached to the finite fields of odd characteristic.  

Let $q=p^m$ be an odd prime power with $p$ an odd prime.  
We denote by $\mathbb{F}_q$ the finite field of $q$ elements.   
Let $H_q$ be the multiplicative table of $\mathbb{F}_q$, i.e., $H_q$ is a $q\times q$ matrix with rows and columns indexed by the elements of $\mathbb{F}_q$ with $(\alpha,\beta)$-entry equal to $\alpha \cdot \beta$. 
Then the matrix $H_q$ is a generalized Hadamard matrix with parameters $(q,1)$ over the additive group of $\mathbb{F}_q$. 
%Letting $G$ be an additively written finite abelian group of order $g$, 
%a square matrix $H=(h_{ij})_{i,j=1}^{g\lambda}$ of order $g\lambda$ with entries from $G$ is called a {\it generalized Hadamard matrix with the parameters $(g,\lambda)$} over $G$ 
%if for all distinct $i,k\in\{1,2,\ldots,g\lambda\}$, the multiset $\{h_{ij}-h_{kj}: 1\leq j\leq g\lambda\}$ contains exactly $\lambda$ times of each element of $G$. 

Let $\phi$ be a permutation representation of the additive group of $\mathbb{F}_q$ defined as follows.  
Since $q=p^m$, we view the additive group of $\mathbb{F}_q$ as $\mathbb{F}_p^m$. 
Again let $U$ be the circulant matrix of order $p$ with the first row $(0,1,0,\ldots,0)$, and a group homomorphism $\phi:\mathbb{Z}_{p}^m\rightarrow GL_{q}(\mathbb{R})$ as $\phi((x_i)_{i=1}^m)= \otimes_{i=1}^m U^{x_i}$.   

From the generalized Hadamard matrix $H_q$ and the permutation representation $\phi$, we construct $q^2$ auxiliary matrices; 
for each $\alpha,\alpha'\in \mathbb{F}_q$, define a $q^2\times q^2$ $(0,1)$-matrix $C_{\alpha,\alpha'}$ to be 
\begin{align*}
C_{\alpha,\alpha'}=(\phi(\alpha(-\beta+\beta')+\alpha'))_{\beta,\beta'\in\mathbb{F}_q}. 
\end{align*}
Letting $x$ be an indeterminate, we define $C_{x,\alpha}$ by $C_{x,\alpha}=J_{q^2}-I_q\otimes J_q$ for $\alpha\in\mathbb{F}_q$. 

It is known that a symmetric Latin square of order $v$ with constant diagonal exists for any positive even integer $v$, see \cite{K}. 
Let $L=(L(a,a'))_{a,a'\in S}$ be a symmetric Latin square of order $q+1$ on the symbol set $S=\mathbb{F}_q\cup\{x\}$ with  constant diagonal $x$. 
Write $L$ as $L=\sum_{a\in S}a\cdot P_a$, where $P_a$ is a symmetric permutation matrix of order $q+1$. 
Note that $P_x=I_{q+1}$. 

From the $(0,1)$-matrices $C_{\alpha,\alpha'}$'s and the Latin square $L$, we construct divisible design graphs \cite{HKM}, that is a symmetric group divisible designs which is adjacency matrices of a graph, as follows. 
Let $R$ be the back identity matrix of order $q^2$.
For $\alpha \in \mathbb{F}_q$, we define a $(q+1)q^2\times (q+1)q^2$ $(0,1)$-matrix $N_{\alpha}$ to be 
\begin{align*}
N_\alpha=(C_{L(a,a'),\alpha}(\delta_{a,a'}I_{q^2}+(1-\delta_{a,a'})R))_{a,a'\in S}=I_{q+1}\otimes (J_{q^2}-I_q\otimes J_q)+\sum_{a\in \mathbb{F}_q} P_a\otimes C_{a,\alpha}R. 
\end{align*}
In order to show that each $N_{\alpha}$ is a divisible design graph and study more properties, we prepare a lemma on $C_{\alpha,\alpha'}$ and $P_a$.  
\begin{lemma}\label{lem:1}
\begin{enumerate}
\item For $\alpha\in\mathbb{F}_q$, $\sum_{a\in\mathbb{F}_q}C_{a,\alpha}=q I_q\otimes \phi(\alpha)+(J_q-I_q)\otimes J_q$. 
\item For $a\in\mathbb{F}_q$ and $\alpha,\alpha'\in\mathbb{F}_q$,
$C_{a,\alpha}C_{a,\alpha'}=q C_{a,\alpha+\alpha'}$.
\item For distinct $a,a'\in\mathbb{F}_q$ and $\alpha,\alpha'\in\mathbb{F}_q$, $C_{a,\alpha}C_{a',\alpha'}=J_{q^2}$.
\item $(J_{q^2}-I_q\otimes J_q)C_{a,\alpha}=C_{a,\alpha}(J_{q^2}-I_q\otimes J_q)=(q-1)J_{q^2}$. 
\item $C_{a,\alpha}R=R C_{a,-\alpha}$. 
%For $\alpha,\alpha',\alpha''\in\mathbb{F}_q$, $(I_q\otimes \phi(\alpha''))C_{\alpha,\alpha'}=C_{\alpha,\alpha'+\alpha''}$. 
\item $\sum_{a,b\in \mathbb{F}_q,a\neq b} P_{a}P_{b}=(q-1)(J_{q+1}-I_{q+1})$. 
\end{enumerate}
\end{lemma}
\begin{proof}
(1): For $\alpha,\beta,\beta'\in\mathbb{F}_q$, the $(\beta,\beta')$-entry of $\sum_{\gamma\in\mathbb{F}_q}C_{\gamma,\alpha}$ is 
\begin{align*}
\sum_{\gamma\in\mathbb{F}_q}\phi(\gamma(-\beta+\beta')+\alpha)&=\begin{cases}\sum_{\gamma\in\mathbb{F}_q}\phi(\alpha) & \text{ if } \beta=\beta' \\ \sum_{\gamma'\in\mathbb{F}_q}\phi(\gamma'+\alpha) & \text{ if } \beta\neq \beta' \end{cases}\\
&=\begin{cases}q\phi(\alpha) & \text{ if } \beta=\beta', \\ J_q & \text{ if } \beta\neq \beta', \end{cases}
\end{align*}
which yields $\sum_{a\in\mathbb{F}_q}C_{a,\alpha}=q I_q\otimes \phi(\alpha)+(J_q-I_q)\otimes J_q$. 

(2):  
For $a,\beta,\beta'\in\mathbb{F}_q$, the $(\beta,\beta')$-entry of $C_{a,\alpha}C_{a,\alpha'}$ is 
\begin{align*}
\sum_{\gamma\in\mathbb{F}_q}\phi(a(-\beta+\gamma)+\alpha)\phi(a(-\gamma+\beta')+\alpha')
&=\sum_{\gamma\in\mathbb{F}_q}\phi(a(-\beta+\beta')+\alpha+\alpha')\\
&=q\phi(a(-\beta+\beta')+\alpha+\alpha'). 
\end{align*}
Thus we have $C_{a,\alpha}C_{a,\alpha'}=q C_{a,\alpha+\alpha'}$. 

(3): It follows from a similar calculation to (ii) with the fact that $ \{(a-a')\gamma\mid \gamma\in\mathbb{F}_q\}=\mathbb{F}_q$. 

(4) and (5) are easy to see, %follows from the fact that each $C_{a,\alpha}$, 
and (6) follows from the equations below. Recall that $S=\mathbb{F}_q\cup\{x\}$. 
\begin{align*}
\sum_{a,b\in \mathbb{F}_q,a\neq b} P_{a}P_{b}&=\sum_{a\in\mathbb{F}_q} P_{a}(\sum_{b\in S\setminus\{x,a\}}P_{b})\\
&=\sum_{a\in\mathbb{F}_q} P_{a}(J_{q+1}-I_{q+1}-P_{a})\\
&=\sum_{a\in\mathbb{F}_q} (J_{q+1}-P_{a}-I_{q+1})\\
&=q(J_{q+1}-I_{q+1})-\sum_{a\in\mathbb{F}_q}P_{a}\\
&=(q-1)(J_{q+1}-I_{q+1}). \qedhere
\end{align*}
\end{proof}

Now we are ready to prove the results for $N_{\alpha}$'s. 
%Note that the result for $N_0$ being a symmetric $((q+2)q^2,q^2+q,q)$-design is well-known, see for example \cite[Exercise~5.7]{S}. 
\begin{theorem}\label{thm:1}
\begin{enumerate}
\item For any $\alpha\in\mathbb{F}_q$, $N_{\alpha}$ is symmetric. 
\item For any $\alpha,\beta\in\mathbb{F}_q$, 
\begin{align*}
N_{\alpha} N_{\beta}
= q^2 I_{q+1}\otimes I_q\otimes \phi(\alpha-\beta)+(q^2-4q+3) I_{q+1}\otimes J_{q^2}+3(q-1)J_{(q+1)q^2}. 
\end{align*}
In particular, $N_{\alpha}^2= q^2 I_{(q+1)q^2}+(q^2-4q+3) I_{q+1}\otimes J_{q^2}+3(q-1)J_{(q+1)q^2}$. %, that is, $N_\alpha$ is the incidence matrix of a symmetric $((q+2)q^2,q^2+q,q)$-design.
\end{enumerate}
\end{theorem}
\begin{proof}
(1): It follows from the properties that the matrices $P_a$ and $C_{a,\alpha}R$ are symmetric for $a\in\mathbb{F}_q$ and $\alpha\in\mathbb{F}_q$. 

(2): We use Lemma~\ref{lem:1} to obtain:  
\begin{align*}
N_{\alpha} N_{\beta} 
&=(I_{q+1}\otimes (J_{q^2}-I_q\otimes J_q)+\sum_{a\in \mathbb{F}_q} P_a\otimes C_{a,\alpha}R)(I_{q+1}\otimes (J_{q^2}-I_q\otimes J_q)+\sum_{b\in \mathbb{F}_q} P_b\otimes C_{b,\beta}R)\\
&=I_{q+1}\otimes(qI_q\otimes J_q+q(q-2)J_{q^2})+\sum_{a\in \mathbb{F}_q} P_a\otimes ((J_{q^2}-I_q\otimes J_q)C_{a,\alpha}R+C_{a,\beta}R(J_{q^2}-I_q\otimes J_q))\\
&\quad +\sum_{a,b\in \mathbb{F}_q} P_{a}P_{b}\otimes C_{a,\alpha}R C_{b,\beta} R\\
&=I_{q+1}\otimes(qI_q\otimes J_q+q(q-2)J_{q^2})+\sum_{a\in \mathbb{F}_q} P_a\otimes 2(q-1)J_{q^2} +\sum_{a,b\in \mathbb{F}_q} P_{a}P_{b}\otimes C_{a,\alpha}C_{b,-\beta}R^2\\
&=I_{q+1}\otimes(qI_q\otimes J_q+q(q-2)J_{q^2})+2(q-1)(J_{q+1}-I_{q+1})\otimes J_{q^2} +\sum_{a,b\in \mathbb{F}_q} P_{a}P_{b}\otimes C_{a,\alpha}C_{b,-\beta}\\
&=q I_{q+1}\otimes I_q\otimes J_q+(q^2-4q+2) I_{q+1}\otimes J_{q^2}+2(q-1)J_{q+1}\otimes J_{q^2} +\sum_{a,b\in \mathbb{F}_q} P_{a}P_{b}\otimes C_{a,\alpha}C_{b,-\beta}.
\end{align*}
The third term of the above is 
\begin{align*}
&\sum_{a\in\mathbb{F}_q} P_{a}^2\otimes C_{a,\alpha}C_{a,-\beta}+\sum_{a,b\in\mathbb{F}_q,a\neq b} P_{a}P_{b}\otimes C_{a,\alpha}C_{b,-\beta}\\
&\sum_{a\in\mathbb{F}_q} I_{q+1}\otimes q C_{a,\alpha-\beta}+\sum_{a,b\in\mathbb{F}_q,a\neq b} P_{a}P_{b}\otimes J_{q^2}\\
&=q I_{q+1}\otimes (q I_q\otimes \phi(\alpha+\beta)+(J_q-I_q)\otimes J_q)+(q-1)(J_{q+1}-I_{q+1})\otimes J_{q^2}\\
&=q^2 I_{q+1}\otimes I_q\otimes \phi(\alpha-\beta)+I_{q+1}\otimes J_{q^2}-qI_{q+1}\otimes I_q\otimes J_q+(q-1)J_{(q+1)q^2}. 
\end{align*}
Therefore 
\begin{align*}
N_{\alpha} N_{\beta} 
%&=q I_{q+1}\otimes I_q\otimes J_q+(q^2-4q+2) I_{q+1}\otimes J_{q^2}+2(q-1)J_{q+1}\otimes J_{q^2}\\
%&\quad +q^2 I_{q+1}\otimes I_q\otimes \phi(\alpha-\beta)+I_{q+1}\otimes J_{q^2}-qI_{q+1}\otimes I_q\otimes J_q+(q-1)J_{(q+1)q^2}\\
&=q^2 I_{q+1}\otimes I_q\otimes \phi(\alpha-\beta)+(q^2-4q+3) I_{q+1}\otimes J_{q^2}+3(q-1)J_{(q+1)q^2}. \qedhere
\end{align*}
%The case of $\alpha=\beta$ follows from $\phi(0)=I_q$. 
\end{proof}

We define $(0,1)$-matrices $A_{\alpha,i}$ ($\alpha\in\mathbb{F}_q,i\in\{0,1\}$) and $A_2$ as 
\begin{align*}
A_{\alpha,0}&=I_{q(q+1)}\otimes \phi(\alpha),\\
A_{\alpha,1}&=N_{\alpha}-I_{q+1}\otimes (J_{q^2}-I_q\otimes J_q),\\
A_{2}&=I_{q+1}\otimes (J_{q^2}-I_q\otimes J_q).
\end{align*}
Note that %$I_{q+2}=P_x$, 
$A_{0,0}=I_{(q+2)q^2}$.  
%Let $\mathbb{F}_q^*=\mathbb{F}_q\setminus\{0\}$.
\begin{theorem}\label{thm:as}
The set of matrices $\{A_{\alpha,0},A_{\alpha,1},A_{2}\mid \alpha\in\mathbb{F}_q\}$ forms a non-commutative association scheme.
\end{theorem}
\begin{proof}
By the definition of $N_{\alpha}$,  $A_{\alpha,1}$'s are non-zero $(0,1)$-matrices such that $\sum_{\alpha\in\mathbb{F}_q}(A_{\alpha,0}+A_{\alpha,1})+A_{2}=J_{(q+1)q^2}$. 
Each of $A_{\alpha,1}$ and $A_2$ is symmetric, and $A_{\alpha,0}^\top=A_{-\alpha,0}$.  
We are now going to show that $\mathcal{A}:=\text{span}_{\mathbb{C}}\{A_{\alpha,i},A_2 \mid \alpha\in\mathbb{F}_q,i\in\{0,1\}\}$ is closed under the matrix multiplication. 
For $\alpha,\beta\in\mathbb{F}_q$, the following are easy to see: 
\begin{align*}
A_{\alpha,0}A_{\beta,0}&=A_{\alpha+\beta,0},\\
A_{\alpha,0}A_{\beta,1}&=A_{\beta,1}A_{\alpha,0}=A_{\alpha+\beta,1},\\
A_{\alpha,0}A_2&=A_2 A_{\alpha,0}=A_2,\\
A_2^2&=(q^2-2q)I_{q+1}\otimes J_{q^2}+q I_{q+1}\otimes I_q \otimes J_q.
\end{align*}
By Lemma~\ref{lem:1}(iv), 
\begin{align*}
A_{\alpha,1}A_2&=A_2 A_{\alpha,1}=\sum_{\alpha\in \mathbb{F}_q}A_{\alpha,1}. 
\end{align*}
Finally by Theorem~\ref{thm:1} (ii), 
\begin{align*}
A_{\alpha,1}A_{\beta,1}&=(N_{\alpha}-A_2)(N_{\beta}-A_2)\\
&=N_{\alpha}N_{\beta}-N_{\alpha}A-A_2 N_{\beta}+A_2^2\\
&=q^2 I_{q+1}\otimes I_q\otimes \phi(\alpha-\beta)+(q^2-4q+3) I_{q+1}\otimes J_{q^2}+3(q-1)J_{(q+1)q^2}\\
&\quad -2(q-1)\sum_{\alpha\in\mathbb{F}_q}A_{\alpha,1}-A_2^2\\
&=q^2 I_{q+1}\otimes I_q\otimes \phi(\alpha-\beta)+q I_{q+1}\otimes(J_q-I_q)\otimes J_q+(q-1)(J_{q+1}-I_{q+1})\otimes J_{q^2}.  
\end{align*} 
Therefore $\mathcal{A}$ is closed under the matrix multiplication. 
\end{proof}

We view $\mathbb{F}_{q}=\mathbb{Z}_p^m$ as the additive group. 
For $\alpha=(\alpha_1,\ldots,\alpha_m),\beta=(\beta_1,\ldots,\beta_m)\in\mathbb{Z}_p^m$, the inner product is defined by $\langle\alpha, \beta\rangle=\alpha_1\beta_1+\cdots+\alpha_m\beta_m$. 
For $\beta\in\mathbb{Z}_p^m$, the irreducible character denoted $\chi_{\beta}$ is $\chi_{\beta}(\alpha)=w^{\langle\alpha, \beta\rangle}$ where $\alpha\in\mathbb{Z}_p^m,w=e^{2\pi\sqrt{-1}/p}$. 
The \defn{character table} $K$ of the abelian group $\mathbb{Z}_p^m$ is a $p^m\times p^m$ matrix with rows and columns indexed by the elements of $\mathbb{Z}_p^m$ with $(\alpha,\beta)$-entry equal to $\chi_{\beta}(\alpha)$. 
Note that $\chi_{\beta}(\alpha)=\chi_{\alpha}(\beta)$. 
Then the Schur orthogonality relation shows that $K K^\top=p^mI_{p^m}$.%, namely $K$ is a Hadamard matrix of order $2^m$. 

To describe the primitive idempotents,  let $F_{\alpha,i}$, $\alpha\in\mathbb{F}_q,i\in\{0,1\}$, be 
\begin{align*}
F_{\alpha,i}&=\sum_{\gamma\in\mathbb{F}_q} \chi_{\alpha}(\gamma)A_{\gamma,i}.  
\end{align*}
%Denote $\mathbb{F}_q^*=\mathbb{F}_q\setminus\{0\}$. 
\begin{lemma}\label{lem:f}
Let $\alpha,\beta\in\mathbb{F}_q$. 
The following hold.
\begin{enumerate}
\item $F_{\alpha,0}F_{\beta,0}=q\delta_{\alpha,\beta} F_{\alpha,0}$.  
\item $F_{\alpha,0}F_{\beta,1}=F_{\alpha,1}F_{\beta,0}=q\delta_{\alpha,\beta} F_{\alpha,1}$.  
\item $F_{\alpha,1}F_{\beta,1}=q^3\delta_{\alpha,\beta} F_{\alpha,0}+\delta_{\alpha,0}\delta_{\beta,0}((q-1)q^2F_{0,1}+q^3 A_2)$.  
\item $F_{0,0}A_2=A_2 F_{0,0}=qI_{q+1}\otimes (J_{q^2}-I_q\otimes J_q)$. 
\item $F_{0,1}A_2=A_2 F_{0,1}=(q^2-q)F_{0,1}$. 
\item If $\alpha\in\mathbb{F}_q^*$, $F_{\alpha,0}A_2=A_2 F_{\alpha,0}=F_{\alpha,1}A_2=A_2 F_{\alpha,1}=O$. 
\end{enumerate}
\end{lemma}

For $i\in\{0,1,2\}, j,k\in\{1,2\},\alpha\in\mathbb{F}_q^*$, let $E_{i},E_{j,k}^{(\alpha)}$ be 
\begin{align*}
E_0&=\frac{1}{(q+1)q^2}J_{(q+1)q^2},\\
E_1&=\frac{1}{(q+1)q^2}((q^2-1)F_{0,0}-(q+1)A_2),\\
E_2&=\frac{1}{(q+1)q^2}(q F_{0,0}-F_{0,1}+q A_2),\\
E_{1,1}^{(\alpha)}&=\frac{1}{q}F_{\alpha,0}, \quad 
E_{2,2}^{(\alpha)}=\frac{1}{q}F_{-\alpha,0}, \quad 
E_{1,2}^{(\alpha)}=\frac{1}{q^2}F_{\alpha,1},\quad  
E_{2,1}^{(\alpha)}=\frac{1}{q^2}F_{-\alpha,1}.
\end{align*}
%From Lemma~\ref{lem:f}, the following is readily obtained. 
\begin{theorem}\label{thm:pi}
Let $S$ be any subset of $\mathbb{F}_q^*$ such that $S\cup (-S)=\mathbb{F}_q^*$ and $S\cap(-S)=\emptyset$. 
The matrices $E_0,E_1,E_2,E_{1,1}^{(\alpha)},E_{1,2}^{(\alpha)},E_{2,1}^{(\alpha)},E_{2,2}^{(\alpha)}$, $\alpha\in S$, provide the Wedderburn decomposition of the adjacency algebra of the association scheme. 
\end{theorem}
\begin{proof}
By Lemma~\ref{lem:f}(i)--(iii), the matrices $E_{i,j}^{(\alpha)}$ satisfy $E_{i,j}^{(\alpha)}E_{j,k}^{(\beta)}=\delta_{\alpha,\beta}\delta_{j,k}E_{i,k}^{(\alpha)}$. 
Also from Lemma~\ref{lem:f}(iv) and (v), it follows that $E_0,E_1,E_2$ are mutually orthogonal idempotents. 
Finally the orthogonality between $E_i$ and $E_{j,k}^{(\alpha)}$ follows from Lemma~\ref{lem:f}(vi).
\end{proof}

\begin{remark}
%The matrix $Q$ is determined by the matrices $Q_k$ below.  
The adjacency algebra of the association scheme in Theorem~\ref{thm:as} is isomorphic to $\oplus_{k=1}^{3+(q-1)/2}\text{Mat}_{d_k}(\mathbb{C})$ where $(d_k)_{k=1}^{3+(q-1)/2}=(1,1,1,2,\ldots,2)$
 with  
\begin{align*}
Q_1&=\begin{pmatrix}\chi_0\\ \chi_{0}\\ 1\end{pmatrix}, 
Q_2=\begin{pmatrix}(q^2-1) \chi_{0}\\ \bm{0} \\ -(q+1)\end{pmatrix},
Q_3=\begin{pmatrix}q \chi_0\\ -\chi_0 \\ q\end{pmatrix},\\
Q_{\alpha}&=(q+1)\begin{pmatrix}q\chi_{\alpha} & \bm{0} & q\chi_{-\alpha} & \bm{0} \\ \bm{0} & \chi_{\alpha}& \bm{0} & \chi_{-\alpha} \\ 0 & 0 & 0& 0\end{pmatrix},
\end{align*} 
for $\alpha\in S$, where $\bm{0}$ denotes the column zero vector. 
\end{remark}

As corollaries of Theorem~\ref{thm:as}, we have the following. 
\begin{corollary}
Let $T$ be the character table of the non-commutative association scheme in Theorem~\ref{thm:as}.
Then 
\begin{align*}
T=\begin{pmatrix} 
\chi_0^\top & q^2\chi_0^\top & q^2-q \\
\chi_0^\top & -q^2\chi_0^\top & -\frac{q^2-q}{q+1} \\
\chi_0^\top & -q\chi_0^\top & q^2-q \\
\chi_{\alpha}^\top+\chi_{-\alpha}^\top & \bm{0}^\top & 0
\end{pmatrix}. 
\end{align*}
\end{corollary}
\begin{proof}
Follows from Proposition~\ref{prop:ct} and Theorem~\ref{thm:as11}. 
\end{proof}
\begin{corollary}
Let $S$ be any subset of $\mathbb{F}_q^*$ such that $S\cup (-S)=\mathbb{F}_q^*$ and $S\cap(-S)=\emptyset$. 
The set of matrices $\{A_{0,0},A_{\alpha,0}+A_{-\alpha,0},A_{0,1},A_{\alpha,1}+A_{-\alpha,1},A_2\mid \alpha\in S\}$ forms a symmetric association scheme with the second eigenmatrix $Q$ 
\begin{align*}
Q=\begin{pmatrix}\chi_0 & (q^2-1)\chi_0 & q\chi_0 &  \frac{q(q+1)}{2}(\chi_{\alpha}+\chi_{-\alpha})  \\ 
\chi_0 & \bm{0} & -\chi_{0}    & \frac{q+1}{2}(\chi_{\alpha}+\chi_{-\alpha}) \\
1 & -(q+1) & q & 0
\end{pmatrix},
\end{align*} 
where $\alpha$ runs over the set $S$. 
\end{corollary}
\begin{proof}
The result follows from Theorem~\ref{thm:BM}.
%The result follows from the fact that $\sum_{i,j=1}^{d_k}E_{i,j}^{(k)}$'s form a basis of the algebra generated by the symmetric $(0,1)$-matrices above, and consist of primitive idempotents.  
\end{proof}

%%%%%%%%%%%%%%%%%%%%%%%%%%%%%%%%%%%%%%%%%%%%%
\section*{Acknowledgement}
Hadi Kharaghani is supported by an NSERC Discovery Grant.  
Sho Suda is supported by JSPS KAKENHI Grant Number 15K21075. 
Th authors thank Keiji Ito for informing us errors in the previous version and the anonymous referee for valuable comments.

\end{document}